\documentclass{sn-jnl}

\usepackage[english]{babel}

\usepackage{amsmath}
\usepackage{graphicx}
\usepackage{subcaption}
\usepackage{tikz}
\usetikzlibrary{shapes.geometric}
\usetikzlibrary{positioning}

\usepackage[
backend=biber,
sorting=nyt,
maxbibnames=99,
url=false,
]{biblatex}
\usepackage{amsmath}
\usepackage{amsthm}
\usepackage{amssymb}
\usepackage{amsfonts}
\usepackage{bbm}
\usepackage[utf8]{inputenc}
\usepackage[dvipsnames]{xcolor}
\usepackage{tikz-cd}
\tikzcdset{scale cd/.style={every label/.append style={scale=#1},
    cells={nodes={scale=#1}}}}
\usepackage{quiver}
\usepackage{tipa}
\usepackage{float}
\usepackage[pdf]{pstricks}
\usepackage{pgf}
\usepackage{scalerel}
\usepackage{algpseudocode}
\usepackage{graphbox}
\usepackage{graphicx}
\usepackage{multirow}
\usepackage{thm-restate}
\usepackage{thmtools}
\usepackage{mathtools}
\usepackage{booktabs}
\usepackage{enumitem}
\usepackage{algpseudocode}
\usepackage{algorithm}

\newtheorem{theorem}{Theorem}[section]

\newtheorem{lemma}[theorem]{Lemma}
\newtheorem{prop}[theorem]{Proposition}
\newtheorem{definition}[theorem]{Definition}

\hypersetup{
    colorlinks=true,
    linkcolor=magenta,
    filecolor=magenta,      
    urlcolor=magenta,
    citecolor = magenta,
    pdftitle={TBD},
    }
    
\newcommand{\R}{\mathbb{R}}

\newcommand{\Z}{{\mathbb Z}}

\newcommand{\PP}{\mathcal{P}}

\newcommand{\cA}{\mathcal{A}}

\makeatletter
\newcommand{\colim@}[2]{%
  \vtop{\m@th\ialign{##\cr
    \hfil$#1\operator@font colim$\hfil\cr
    \noalign{\nointerlineskip\kern1.5\ex@}#2\cr
    \noalign{\nointerlineskip\kern-\ex@}\cr}}%
}
\newcommand{\colim}{%
  \mathop{\mathpalette\colim@{\rightarrowfill@\textstyle}}\nmlimits@
}
\makeatother

\renewcommand{\phi}{\varphi}
\newcommand{\inv}{^{-1}}

\DeclareMathOperator{\Lk}{\mathrm{Lk}}
\DeclareMathOperator{\St}{\mathrm{St}}

\DeclareMathOperator{\lSt}{\mathrm{St}^{\downarrow}}

\let\Im\relax 
\DeclareMathOperator{\Im}{Im}
\DeclareMathOperator{\Ker}{Ker}

\newcommand{\colex}{<^{\mathrm{co}}}

\newcommand{\mycomment}[1]{}

\DeclareMathOperator*{\argmax}{arg\,max}
\DeclareMathOperator*{\argmin}{arg\,min}

\begin{document}

\title{Updating a Discrete Morse Vector Field for a Lower Star Filtration Vineyard}
\author[1]{\fnm{Kevin} \sur{Woytowich}}\email{woytow1@msu.edu;}

\author[2]{\fnm{Nkechi} \sur{Nnadi}}\email{nnadinke@msu.edu}

\author*[1,2]{\fnm{Elizabeth} \sur{Munch}}\email{muncheli@msu.edu}

\affil[1]{\orgdiv{Department of Mathematics}, \orgname{Michigan State University}, \orgaddress{\city{East Lansing}, \postcode{48824}, \state{Michigan}, \country{USA}}}

\affil[2]{\orgdiv{Department of Computational Mathematics, Science, and Engineering}, \orgname{Michigan State University}, \orgaddress{\city{East Lansing}, \postcode{48824}, \state{Michigan}, \country{USA}}}

\date{}

\abstract{In this paper, we provide a construction of an acyclic discrete vector field that is compatible with a total order associated with a lower star filtration on a simplicial
complex, called the colex vector field. We show that the colex vector field induces a filtered acyclic vector field, whose resulting Morse complex computes the persistent homology of the lower star filtration on the underlying simplicial complex. We show that the colex vector field can be recomputed quickly when the vertex function that induces the lower star filtration is modified via order-adjacent vertex swaps. We provide a framework for storing and computing the number of paths between cells in the simplicial complex, as well as a method to update these values quickly when the colex vector field changes. Finally, we provide publicly available proof-of-concept code for the ideas shown. When applying it to the Persistent Homology Transform, we show that its runtime is
comparable to a more standard matrix reduction approach.}


\maketitle

\section{Introduction}
Persistent homology \cite{boundary, Zomorodian2005}, as the flagship tool in topological data analysis (TDA), involves tracking the appearance, sustenance, merging, and disappearance of homology features across an ordered sequence of resolutions of a topological space known as a \emph{filtration}. 
This summary is recorded as a collection of points in $\R^2$ tracking the appearance and disappearance of features known as a persistence diagram. One of the defining strengths of persistent homology is the stability theorem for persistence diagrams \cite{CohenSteiner2007Stability, Bauer_stability, bubenik2015metrics,chazal2009Stability}, which shows that the bottleneck distance between persistence diagrams is bounded by the magnitude of the perturbation of the filtering function. 
Consequently, persistence provides a mathematically robust summary of topological structure that is resilient to noise and small deformations. 
Over the past two decades, persistent homology (and TDA in general) has found applications across a remarkably diverse range of disciplines, such as describing shape in dynamical systems \cite{TymochkoMunch}, atmospheric science \cite{TYMOCHKO2020137}, electoral redistricting and polling access \cite{Moon2022, Hickok_polling}, among many other applications. 
Interested readers can find a curated collection of practical uses in the DONUT database \cite{DONUT}.
The foundational work of \cite{Zomorodian2005} expounds on the algebraic framework underlying persistent homology, and \cite{simplification} introduces the standard persistence reduction algorithm for computing persistence over the field $\Z_2$. 
This algorithm is based on reducing the boundary matrix of a filtered simplicial complex and has worst-case time complexity $O(N^3)$, where $N$ denotes the number of simplices in the filtration.

Several efforts have been made to improve the standard matrix reduction algorithm, yielding significant gains in practical performance even when theoretical asymptotic complexity remains unchanged. 
The clear-and-compress algorithm \cite{clear_compress} partitions the boundary matrix into local chunks, compresses unresolved columns, and exploits parallelism to substantially reduce computation in practice, despite preserving the worst-case complexity of the standard algorithm. 
In \cite{sparse}, the authors introduced new reduction strategies based on column swapping and retrospective reductions that preserve sparsity and substantially improve practical performance. Building on the fact that ordinary persistence runs in time $O(M(N))$ \cite{MatrixMult}, where $M(N)$ is the time it takes to multiply two $N\times N$ matrices,  \cite{MorozovZigzagMatrixMult} presents an algorithm that computes zigzag persistence over a finite field in time $O(M(N) + N^2\log^2 N)$. 
As far as we know, the first rigorous average-case complexity analysis for matrix reduction is presented in \cite{giunti2025}, thereby complementing work on practical algorithmic improvements. 
Among the most influential advances in persistent homology software is Ripser \cite{Bauer2021}, which achieves remarkable computational efficiency through a combination of cohomological methods and optimized matrix reduction techniques. 
This illustrates the ongoing effort to accelerate persistence computations for the special case of the Rips filtration of a point cloud. 

In practice, however, data is often not static and naturally evolves with respect to an external parameter; e.g.~time-varying measurements, dynamical systems, video streams, biological processes, and sequential point clouds arising from moving objects. 
Such settings naturally give rise to families of filtrations motivating the study of how a topological summary such as the persistence diagram evolves across the parameter space. 
Different choices of parameterization give rise to different families of topological summaries. 
When the parameter is one-dimensional, that is, indexed by $\R$, the continuous evolution of persistence diagrams is known as a persistence vineyard \cite{dmitriy}.
The persistent homology transform (PHT) provides an example of a family of persistence diagrams parameterized over the sphere $\mathbb{S}^{d-1}$ \cite{pht}.
Parameterizations over arbitrary base spaces form persistence diagram bundles \cite{abbyhickok1}, and if more simplicial information is retained than only the points in the persistence diagrams, these induce canopies with a natural topological structure \cite{GiuntiMunch2026}.
Other parameterized topological summaries include the Euler characteristic transform (ECT) \cite{Munch02012025}, obtained by indexing Euler characteristic curves induced by height functions over directions, and CROCKER plots, which summarize families of Betti curves \cite{crocker_plots2015, crocker_plots2022}. 
These provide descriptors of shape that are, in some cases, stable under certain conditions \cite{Marsh2026, SELECT_Stability}, like the PHT \cite{Curry2022}, although they generally capture less topological information than persistence-based summaries.

Persistence pairs record the birth and death of topological features across a filtration, and are represented as points in the persistence diagram. 
By the stability theorem, the points on the diagram represented by persistence pairs evolve continuously with the parameter, tracing out a collection of trajectories which are referred to as \emph{vines}\cite{dmitriy, Morozovthesis}. 
Vineyards have been applied to effectively describe dynamic changes in musical compositions \cite{Bergomi2020}, protein folding \cite{dmitriy}, dynamical system regimes \cite {Algar_etal}, polymer chains \cite{Bertoglio2025}, and functional connectivity of the brain as seen in EEG data \cite{yoon_etal} and fMRI paradigms \cite{10.1371/journal.pone.0255859, Abdallah_etal}. 
These collections of persistent homology summaries provide a robust way to discriminate between shapes. 
Recent work has expanded the mathematical foundations of persistence vineyards beyond their computational origins \cite{turner2023}. 
Chambers et al.~\cite{braiding_vineyards} establish a connection between vineyards and knot theory by showing that closed vineyards can realize arbitrary links, highlighting the expressive topological complexity of vineyards.

A natural concern that arises when examining families of topological summaries rather than individual ones is the computational cost. 
The continuity guaranteed by the stability theorem suggests that successive diagrams in a smoothly varying filtration should differ only locally, motivating algorithms that update existing persistence computations rather than recomputing them from scratch. 
\cite{dmitriy} provides procedures for updating persistence pairs after different cases of a simplex order change in the filtration. 
In \cite{giunti2025pruning}, the authors develop an algorithm for updating persistence barcodes and representative cycles following simplex removals. 
Their \emph{simplicial removal update procedure} incrementally updates the reduced boundary matrix, avoiding recomputation from scratch while minimizing the number of matrix column operations.
More work is done in \cite{Tamal_speedup, tamal_paper} to develop algorithms for tracking representative generating cycles with temporal coherence, allowing essential homology classes to be followed throughout a dynamic filtration. 
More recently, \cite{jose_paper} proposed \emph{move schedules}, which replace long sequences of adjacent transpositions with more efficient coarse moves. 
This reduces the overall cost of dynamic persistence computations by decreasing the number of filtration updates that must be processed. 

Our work addresses the complementary problem of how to efficiently update the persistence decomposition following a local update. 
We answer this question by exploiting discrete Morse theory to minimize the computational effort required after changes in filtration.
Researchers have long known that discrete Morse theory is an effective tool for accelerating persistent homology computations \cite{Gunther2011, Gunther2012, Kannan2019, Bauer2011}. 
In \cite{vidit_paper}, the authors introduce a method of using combinatorial discrete Morse theory (DMT) that can dramatically accelerate persistent homology computations by reducing the size of the complex before persistence is computed, while retaining the persistent homology intact. 
Yang et al.~\cite{Yang2025} accelerate the initial persistent homology computation through Morse theory-inspired matching and matrix reductions. 
Morse theory has been extended to zigzag persistence \cite{Maria2019}, as well as networks and hyper-networks \cite{Saucan2021}. 
Not only that, \cite{tripart} demonstrates why simplex orderings and matrix reductions carry rich algebraic information. 
They show that for any ordered cell complex, matrix reduction induces a canonical tri-partition of cells into tree, co-tree, and homology generators, yielding canonical bases for cycle and boundary groups. 

We build upon the foregoing ideas in our work to employ discrete Morse theory to accelerate dynamic updates of persistence pairs under evolving filtrations. 
We specifically consider \emph{lower-star filtrations}, given by a real-valued function on the vertices of a simplicial complex, with its higher-dimensional simplices entering the filtration only when the participating vertex with the highest function value has been included. 
We choose to impose a colexicographic ordering on all the simplices in the filtration, so that the ordering changes when the vertex function changes. 
Using this setup, we characterize when changes in simplex order correspond to local changes in the Morse matching, and how this updates the persistence. 
We demonstrate that only a small subset of critical cells need to be reconsidered after each update, and that persistence pairs obtained via our discrete Morse update procedure coincide with those produced by the standard matrix reduction algorithm following local filtration changes. 
To support the theory, we provide a proof-of-concept implementation to demonstrate its practicality. 
Preliminary experimental results show that our implementation consistently outperforms recomputation via the standard reduction algorithm on a collection of benchmark examples, providing initial evidence that the proposed framework can yield computational savings. 

\section{Background}

In this section, we give the necessary background on posets, persistent homology, lower star filtrations, and discrete Morse theory to be able to state our results. 

\subsection{Basic Notions}

We start by using an ordering on a set to give an ordering on words of the set, which we will use to create an ordering of simplices from a vertex ordering.

\begin{definition}
\label{defn:totalorder}
    A totally ordered set $(\cA,\leq)$ is a set such that for all $x, y \in \cA$, either $x\leq y$ or $y \leq x$. 
    For subsets $A, B \subseteq \cA$, we abuse notation to write $A < B$ if $a<b$ for all $a \in A$ and $b \in B$. 
\end{definition}

\begin{definition}
\label{def:colex}
Assume we are given a finite, totally ordered set $(\cA,\leq)$. 
The \emph{words} of $A$ are finite sequences of elements of $\cA$.
\emph{Colexicographic (colex) order} on the words of $\cA$ is a total order, $\colex$, given as follows for two words $a = a_0a_1\cdots a_k $ and ${b = b_0b_1\cdots b_{\ell}}$. 
Let $I =\{0, 1, \ldots, \min\{k,l\}\}$ and $J=\min\{i\in I: a_{k-i}\neq b_{\ell-i}\}$. We say that $a \colex b$ if $a_{k-J} <b_{\ell-J}$ for the first index $i \in \{0,\cdots,J\}$  where $a_{k-i} \neq b_{\ell-i}$.
Moreover, if $a_{k-i} = b_{\ell-i}$ for all $i \in I$ and $k<\ell$, then $a \colex b$. 
\end{definition}

In the rest of this paper, we shall refer to this ordering as colex order. Note that this definition is set up so that moving in from the end of the two words, we look for the first entry that differs to decide on the ordering; for example $hands \colex shards$ using alphabetical order since $n$ is before $r$ in the alphabet.
Because we will often be interested in just the ending of a word, we write $w = \cdots \alpha$ if the word $w$ ends with the word $\alpha$.

\begin{definition}
    A simplicial complex, $K$, is a (finite) set of subsets of a vertex set $V$ that is closed under the subset relation. 
    That is, if $\sigma \subseteq \tau$ for some $\tau \in K$, then $\sigma \in K$.
    If $\sigma \subseteq \tau$, we refer to $\sigma$ as a \emph{face} of $\tau$. 
    The face relation is denoted $\sigma \leq \tau$ with no subscripts;  if we want to emphasize that $\sigma \neq \tau$, we use $\sigma < \tau$.
\end{definition}

Throughout, we assume a colex order on the simplices of $K$, inherited by viewing simplices as words of $V$ with an injective function $f \colon V \rightarrow \R$ inducing the totally ordered set $V$ for Definition \ref{def:colex}.
Moreover, given a function $f \colon V \rightarrow \R$, we extend it to all of $K$ by defining $f(\sigma) = \max\{f(v) \mid v \in \sigma\}$. 
We say $f \colon K \rightarrow \R$ is \textit{induced by a vertex function}.
Note that for such a function, $\sigma \leq \tau$ (the face relation) implies $f(\sigma) \leq f(\tau)$, and in particular, the sublevel sets $f^{-1}((-\infty, r])$ are subcomplexes of $K$ for each $r\in \R$.

\begin{definition}
\label{defn:compatibleOrdering}
    Given a function $f:K \to \R$, a \emph{compatible ordering} $\prec$ of the simplices of $K$ is a total order such that if 
    \textit{(i)} $\sigma < \tau$ (face relation) or 
    \textit{(ii)} $f(\sigma) < f(\tau)$, then $\sigma \prec \tau$.
\end{definition}

For example, when $f$ is induced by a vertex function, the corresponding colex order is a compatible ordering of the simplices of $K$.

\begin{definition}
    \label{def:starsAndLinks}
    For any vertex $v\in V$,
    
    \begin{itemize}
        \item The \emph{star} of $v$ is the set $St(v) = \{\sigma  \in K \mid v \in \sigma \}$.  
        \item The \emph{closed star} of $v$ is the set $\overline \St(v) \coloneq \{ \sigma \in K \mid \sigma \leq \tau \in \St(v)\}$. Equivalently, this is the smallest subcomplex of $K$ that contains $St(v)$.
        \item The \emph{link} of $v$ is the set $Lk(v) = \{ \sigma  \in K: v\notin \sigma, \: \{v\}\cup \sigma \in K\}$.
        \item The \emph{lower star} of $v$ is $\lSt(v) = \{\sigma \in \St(v) \mid f(\sigma) \leq f(v) \}$. Equivalently, this can be written as $\lSt(v)= \{\sigma \in \St(v) \mid f(w) \leq f(v), \; \forall w \in \sigma  \}$. 
    \end{itemize}
\end{definition}

We emphasize that the star and link of a vertex are generally not simplicial complexes, since they do not necessarily satisfy the closure property. Consequently, these constructions should be regarded simply as collections of simplices rather than subcomplexes.


\begin{lemma}
\label{lem:closedStarPartition}
For any vertex $v \in V$, 
\begin{enumerate}[label=(\alph*)]
    \item $\sigma \in \overline \St(v)$ and $v \in \sigma$ if and only if $\sigma \in \St(v)$, 
    \item $\sigma \in \overline \St(v)$ and $v \notin \sigma$ if and only if $\sigma \in \Lk(v)$, 
    \item $\overline \St(v) = \St(v) \sqcup \Lk(v)$.
\end{enumerate}
\end{lemma}

\begin{proof}$ $

\begin{enumerate}[label=(\alph*)]
    \item If $v \in \sigma$, then $\sigma \in \St(v)$ by definition.
    Conversely, if $\sigma \in \St(v)$, then $v \in \sigma$ by definition; moreover, $\sigma \leq \sigma \in \St(v)$, so $\sigma \in \overline \St(v)$ trivially.

    \item If $\sigma \in \overline \St(v)$, then $\sigma \leq \tau \in \St(v)$; 
    by definition, $v \in \tau$, and therefore it has the face $\{v\} \cup \sigma \in K$.
    However, since $v \notin \sigma$, this implies that $\sigma \in \Lk(v)$.
    Conversely, $\sigma \in \Lk(v)$ directly implies that $v \notin \sigma$, and $\sigma \leq \{v\} \cup \sigma \in \St(v)$, so $\sigma \in \overline \St(v)$.

    \item By the previous two results, partitioning $\overline \St(v)$ into $\{\sigma \mid v \in \sigma \}$ and $\{\sigma \mid v \notin \sigma\}$ is the exact same as partitioning it into $\St(v)$ and $\Lk(v)$. \qedhere
\end{enumerate}
\end{proof}

\begin{definition}
    A filtration of a simplicial complex $K$ indexed by a finite set ${I = \{i_0 \leq i_1 \leq \dots \leq i_n \} \subset \R}$ is a collection of subcomplexes $\mathcal{K} = \{ K_i \mid i \in I\}$ with $K_i \subseteq K$ for each $i$ such that $i \leq j$ implies $K_i \subseteq K_j$.
    Additionally, we require $K_{i_0} = \emptyset$ and $K_{i_n} = K$.
    When the exact values of $I$ are irrelevant, a filtration will be denoted \[\{0 = K_{0} \subseteq K_{1} \subseteq  \ldots \subseteq K_{n} = K\}.\]
    Each $K_i \in \mathcal{K}$ may be referred to as a frame.
\end{definition}

A common construction is a sublevelset filtration, where given a function ${f\colon K \to \R}$, $K_i = f\inv(-\infty,a_i]$ for a collection $a_0 <a_1<\cdots a_n$. 
To ensure these sublevelsets are indeed simplicial complexes, we assume the function $f$ is \textit{monotone}, meaning that if $\sigma \leq \tau$, then $f(\sigma) \leq f(\tau)$. 
The \emph{lower star filtration} is the sublevelset filtration obtained from a function $f\colon K \to \R$ induced by a vertex function, indexed by $f(V)$.
In a lower star filtration, each vertex $v$ enters the filtration along with its lower star. 
This means, in particular, that the simplices entering the filtration at that function value each have $v$ as the participating vertex with the highest function value. 
This also means that any simplices that have vertices $w_i$ with $f(w_i) > f(v)$ do not get added at this time. 
In this context, where sublevel sets are induced by a function $f$, we refer to $f$ as a \emph{filter function}.


\subsection{Persistent homology}
\label{sec:PersistentHomology}
Here we review the relevant background on persistent homology for our special case of a lower star filtrationand direct the interested reader to \cite{oudot, DeyWang_text, edelsbrunner_text} for details. Let $K$ be a simplicial complex and $f:K \to \R$ be induced by a vertex function. Fix a compatible ordering and relabel the simplices $\sigma_1, \sigma_2, \ldots, \sigma_n$, such that $i<j$ implies $f(\sigma_i) \leq f(\sigma_j)$. The \emph{$p$-chain group} of $K$, $C_p(K)$, is the group whose elements are formal sums of $p$-simplices with coefficients in some field, which in our case is $\Z_2$. The map $\partial_p: C_p \rightarrow C_{p-1}$ defined as 
\begin{equation} \partial_p(\sigma) = \sum\limits_{i=0}^p\{v_0, v_1, \ldots, \hat{v}_i, \ldots, v_p\},
\end{equation} for a $p$-simplex $\sigma = \{v_0, v_1, \ldots, v_p\}$, where $\hat{v_i}$ denotes the omission of that vertex, is known as the \emph{$p^{th}$ boundary map}. This map is a homomorphism and satisfies the relation $\partial_p \circ \partial_{p+1} =0$. So we can obtain a sequence of chain groups with boundary homomorphisms $\partial_p$ between them, known as a \emph{chain complex}:
$$0\rightarrow \ldots \xrightarrow[]{\partial_{p+2}} C_{p+1} \xrightarrow[]{\partial_{p+1}} C_{p} \xrightarrow[]{\partial_{p}} C_{p-1} \xrightarrow[]{} \ldots \xrightarrow[]{\partial_1} C_0\xrightarrow[]{\partial_0} 0.$$

A $p$-chain whose boundary is empty is a \emph{cycle}. A $p$-chain which is the boundary of any $(p+1)$-chain is a \emph{boundary}. The group of $p$-cycles is the kernel of the boundary homomorphism $\partial_p$, and similarly, the group of $p$-boundaries is the image of $\partial_{p+1}$. We denote these groups as $Z_p(K)$ and $B_p(K)$, respectively.

\begin{definition}\label{defn:Homology_group}
    The $p^{th}$ homology group of $K$, $H_p(K)$ is defined as the quotient group $ \dfrac{Z_p(K)}{ B_p(K)} = \dfrac{\Ker \partial_p}{\Im \partial_{p+1}}$, and its rank is called the \emph{$p^{th}$ Betti number}.
\end{definition}

We then consider this with respect to the lower star filtration ${\{0 = K_0 \subseteq K_1 \subseteq \ldots \subseteq K_n =K \}}$ of the simplicial complex $K$.  The inclusions $K_i \subseteq K_j$ induce linear transformations $F_p^{i\leq j}: H_p(K_i) \rightarrow H_p(K_j)$ between the respective $p^{th}$ homology groups. Thus, for each homology dimension $p$,  we obtain a sequence of homology groups 
\begin{equation} 0 \rightarrow H_p(K_1) \rightarrow H_p(K_2) \rightarrow \ldots \rightarrow H_p(K_n), \end{equation} known as a \emph{persistence module}. 

\begin{definition}
    We define the \emph{$p^{th}$ persistent homology groups} to be the images of induced homomorphisms $H_p^{i\leq j} = \Im \; F_p^{i\leq j}$. The associated persistent \emph{Betti number} is the rank of the image, written as $\beta^p_{i\leq j} = \text{rank } \: H_p^{i\leq j}$. 
\end{definition}
The collection of Betti numbers may be visualized on a \emph{persistence diagram}, $\mathcal{D}^p(f)$, which is a multiset of points in the extended Euclidean plane $(\R \cup \{\infty\})^2$ where each element $(c_i,c_j) $ has multiplicity $\mu^p_{i\leq j}$, with $$\mu^p_{i\leq j} = (\beta^p_{i\leq j-1} - \beta^p_{i\leq j}) - (\beta^p_{i-1\leq j-1} - \beta^p_{i-1\leq j}).$$
When the dimension $p$ is clear from the context, we will simply refer to $\mathcal{D}(f)$ as the associated persistence diagram. We shall denote the space of persistence diagrams as $\mathrm{DGM}$, and describe a metric on this space to allow us to define a topology and thus have a well-defined notion of continuity. 

\begin{definition}[Bottleneck metric]
\label{defn:BottleneckDistance}
Fix $q \in [1,\infty]$, and let $\Delta = \{(x,x)\in \R^2\}$ denote the diagonal. Given a pair of persistence diagrams $X$ and $Y$, considered as multisets of points in the interior upper half space, $\left\{(x,y) \in \R \times \overline{\R} \mid x<y \right\}$, a \emph{partial matching} is a bijection
$\phi: X' \to Y'$ for subsets $X' \subseteq X$ and $Y' \subseteq Y$. The cost of a partial matching is given by
$$ c(\phi) =\max\left\{
\sup_{x\in X'} \lVert x-\phi(x)\rVert_q,\,
\sup_{z\in (X\setminus X')\cup(Y\setminus Y')}
\lVert z-\Delta\rVert_q
\right\}.
$$
The \emph{bottleneck metric} between $X$ and $Y$ is
\[ W_{\infty}(X,Y) = \inf_{\phi} c(\phi), \]
where the infimum is taken over all partial matchings of $X$ and $Y$.
\end{definition}

The stability theorem in \cite{CohenSteiner2007Stability, Chazal2016} establishes the robustness of persistence diagrams under small perturbations of the filter function $f$ using the bottleneck metric, making them a reliable descriptor of topological features in filtered spaces.

\subsection{Matrix reduction}
We briefly review the standard matrix reduction algorithm following \cite{boundary} for computing persistent homology over $\Z_2$, as it serves as the baseline against which we measure the efficiency of our proposed method. Consider the incidence matrix $D$, of simplices in $K$, which is a binary matrix of size $|K|\times |K|$ such that 
$$ D[i,j] = \begin{cases}
        1 & \text{if } \sigma_i \subseteq \sigma_j\\
        0 & \text{otherwise}.
    \end{cases}$$
The following definition arises from the Algorithm \ref{alg:MatrixReduction} introduced in \cite{boundary}, which proceeds by column-reducing $D$ into another binary matrix, $R$. 
Define the function ${low_R \colon \{1,\ldots, |K|\} \rightarrow \{1,\ldots, |K|\}}$ such that $low_R(j)$ is the row index of the lowest $1$ in column $j$ of $R$, and is undefined if the column consists entirely of zeros. 

\begin{definition}
    The matrix $R$ is said to be \emph{reduced} if $low_R(j)\neq low_R(j')$ whenever $j\neq j'$ and columns $j$ and $j'$ are non zero columns. That is, no two columns share the same row index for their lowest $1$ entries.
\end{definition}

\begin{algorithm}[h!]
\caption{Matrix reduction algorithm}
\label{alg:MatrixReduction}
\begin{algorithmic}[1]
    \State $R= D$;
    \For{$j = 1$ to $n$}
        \While{there exists $j' < j$ with $low(j') = low(j)$ and $low_R(j) \neq -1$}
            \State add column $j'$ to column $j$
        \EndWhile
        \If{$low_R(j) \neq -1$}
            \State $i:=low_R(j)$; generate pair $(\sigma_i, \sigma_j)$
        \EndIf
    \EndFor
\end{algorithmic}
\end{algorithm}

The non-zero columns of a reduced matrix are linearly independent over $\Z_2$. The reduced matrix is computed as in the pseudocode of Algorithm \ref{alg:MatrixReduction}.
This algorithm has cubic time complexity with respect to the number of simplices, $n$. Essentially, the matrix reduction algorithm decomposes $D$ as a product $RU$, where $R$ is the reduced matrix and $U$ is an invertible upper-triangular matrix.

\begin{definition}
Given a reduced matrix $R$ and an invertible upper triangular matrix $U$ such that $D=RU$, we define $\sigma_i$ to be a \emph{birth simplex} if there exists a column $j$ such that $i=low_R(j)$ or, equivalently, if the reduced column $R[i]$ is zero. The simplex $\sigma_j$ such that $i=low_R(j)$, or equivalently, whose reduced column $R[j]$ is nonzero, is referred to as a \emph{death simplex}.
A pair $(\sigma_i, \sigma_j)$ of these simplices is referred to as a \emph{birth-death pair} or \emph{persistence pair}.
\end{definition}

By the Pairing Uniqueness Lemma in \cite{dmitriy}, the birth-death simplex pairs are unique, regardless of the particular $RU$ decomposition. We also note that the persistence diagram is uniquely determined by the underlying persistence module. Therefore, any compatible ordering as in Definition \ref{defn:compatibleOrdering} gives the same persistence diagram, even if the pairing is different. This invariance of the resulting persistence diagram rests on the fact that the multiplicities can be computed using only the ranks of images of the inclusion-induced linear transformations, completely independent of any basis or matrix algorithm.

\subsection{Parameterized persistence}
We now focus on the case of updating persistence pairs induced by a continuously evolving filter function on a fixed simplicial complex $K$, following \cite{abbyhickok1, GiuntiMunch2026}.

\begin{definition}
A \emph{parameterized filter function} is a function $f : K \times B \to \mathbb{R}$, where $B$ is a topological space, the map
\[ \begin{matrix}
    f_p \colon & K &\to &\mathbb{R}, \qquad\\
&\sigma& \mapsto &f(p,\sigma)
\end{matrix} \] is monotone for every $p \in B$, and $f_\sigma \colon B \to \mathbb{R};\; p \mapsto f(p,\sigma)$ is continuous for every $\sigma \in K$.
\end{definition}

The monotone assumption means that each fiber of $f$ induces a filtration, which in turn yields a persistence module, and thus a persistence diagram.
The result is the following map. 

\begin{definition}
Given a fibered filtration function
$f : K \times B \to \mathbb{R}$, 
the \emph{induced persistence map} is
\[
\begin{matrix}
G \colon & B &\longrightarrow &\mathrm{DGM} \\
      & p &\longmapsto &\mathcal{D}(f_p).
\end{matrix}
\]
\end{definition}

\begin{theorem}
For a fibered filtration function $f: K \times B \to \mathbb{R}$, the induced persistence map $G$ is uniquely defined and is continuous in the bottleneck metric.
\end{theorem}
\begin{proof}
    Uniqueness follows from the fact that the persistence diagram $\mathcal{D}(f_p)$ is uniquely determined by the values of $f_p$ for each $p\in B$. Continuity is given by the stability theorem as in \cite{CohenSteiner2007Stability, Chazal2016}.
\end{proof}

One special case of these parameterized persistence structures is the \emph{vineyard} \cite{dmitriy}, which arises from a $1$-parameter family of filtrations of simplicial complexes, with $B=\R$.
For each parameter value $t \in \R$, persistent homology yields a persistence diagram $\mathcal{D}(f_t)$. The collection $\{\mathcal{D}(f_t)\}_{t\in\mathbb{R}}$
forms a persistence vineyard, in which the birth-death points of the persistence diagrams evolve continuously as the parameter changes. The trajectories traced by these points, called \emph{vines}, record the evolution of individual homological features across the family of filtrations, providing a dynamic representation of topological changes over time or with respect to another continuously varying parameter.  

Another special case arises by considering $B=\mathbb{S}^{d-1}$ when $K\subseteq \R^d$. The \emph{directional transform} is the parameterized filtration function defined by
\[
\begin{matrix}
f \colon  & K \times \mathbb{S}^{d-1}  &\longrightarrow &\mathbb{R}\\
& (\sigma,\omega)& \longmapsto &\max_{v\in\sigma}\langle v,\omega\rangle,
\end{matrix}
\]
where each vertex $v$ is identified with its embedding in $\mathbb{R}^d$, and $\langle \cdot, \cdot \rangle$ is the standard inner product on $\R^d$. \emph{The Persistent Homology Transform (PHT)} of $K$ is then the induced persistence map
\[
\begin{matrix}
\operatorname{PHT}(K) \colon & \mathbb{S}^{d-1}&\longrightarrow &\mathrm{DGM},\\
&\omega&\longmapsto &\mathcal{D}(f_\omega),
\end{matrix}
\]
which assigns to each direction the persistence diagram of the corresponding directional filtration. The PHT, introduced by Turner et al.~\cite{pht}, is among the most prominent applications of parameterized persistence. A fundamental reason for its importance is that, under suitable conditions, the collection of persistence diagrams obtained from all directions uniquely determines the underlying topological structure for a broad class of embedded shapes \cite{pht, Curry2022, Ghrist2018}.

\subsection{Discrete Morse Theory}

This work uses discrete Morse theory as developed by Forman \cite{Forman}, which adapts techniques and results from Morse theory on smooth complexes to cell complexes. 
This will provide results about the homology of a simplicial complex while only needing to compute it for a smaller complex.
Here we follow \cite{Forman, knudson2015morse, vidit_paper, filtered} for the basic theory. 
Throughout, let $K$ be a simplicial complex with vertex set $V$, where $\sigma^{(p)}$ denotes a $p$-dimensional simplex in $K$.

\begin{definition}[{\cite[Definition 6.14]{knudson2015morse}}]
    A \textbf{discrete vector field} $F$ on $K$ is a collection of pairs of simplices in $K$ of the form $\{\alpha^{(p)} < \beta^{(p+1)}\}$, with no simplex in more than one pair.
\end{definition}

The pair $\{\alpha, \beta\}$ can be thought of as drawing an arrow from $\alpha$ to $\beta$, henceforth denoted $\{\alpha \rightarrow \beta\}$. 
In this case, $\alpha$ will be referred to as the \textit{tail} of that arrow, and $\beta$ will be referred to as the \textit{head} of that arrow.
If a simplex $\sigma$ is not part of any arrow of $F$, then we say that $\sigma$ is \textit{critical}.\footnote{We mention that standard practice is to denote a discrete vector field with $V$ rather than $F$. In this paper, $V$ is reserved for the vertex set of $K$.}

\begin{definition}[{\cite[Definition 6.17]{knudson2015morse}}]
\label{def:FPath}
Let $F$ be a discrete vector field on $K$. 
An \textbf{$F$-path} is a sequence of simplices 
\[ \alpha_0^{(p)}, \beta_0^{(p+1)}, \alpha_1^{(p)}, \beta_1^{(p+1)}, \alpha_2^{(p)}, \dots, \beta_r^{(p+1)}, \alpha_{r+1}^{(p)}\]
such that for each $i$, we have $\{\alpha_i, \beta_i\} \in F$, $\beta_i > \alpha_{i+1}$, and $\alpha_i \neq \alpha_{i+1}$.
\end{definition}
To highlight its structure, we will often denote an $F$-path as 
\[\alpha_0 \rightarrow \beta_0 > \alpha_1 \rightarrow \dots >\alpha_{r+1}.\]
A \textit{nontrivial} $F$-path additionally requires $r \geq 0$, and an \textit{closed} $F$-path requires ${\alpha_0 = \alpha_{r+1}}$.
If $F$ has no nontrivial closed $F$-paths, then we say it is \textit{acyclic} (equivalently known as a \textit{gradient}, as in \cite{Forman}).
We also mention for clarity that we often consider $F$-paths that start at $\beta_0^{(p+1)}$ rather than $\alpha_0^{(p)}$; the important behavior is that the sequence alternates between two consecutive dimensions.

The seminal result of Forman \cite{Forman} is that an acyclic discrete vector field $F$ on a simplicial complex $K$ induces a complex with the same homology as $K$. 
We will work in the $\Z/2\Z$ coefficient ring, in order to avoid complications of orientation.
Following \cite{Forman} and \cite{King_mod2}, let $F$ be an acyclic discrete vector field on $K$, and for each $p$, let $C_p(K)$ be the space of chains of $p$-simplices, and $\mathcal{M}_p$ be the span of $p$-simplices that are critical in $F$.
For each $p$, let $\partial_p \colon \mathcal{M}_p \rightarrow \mathcal{M}_{p-1}$ be given by 
\[ \partial_p(\tau) \coloneq \sum_{\sigma \in \mathcal{M}_{p-1}} \Gamma(\tau, \sigma) \sigma, \]
where $\Gamma(\tau, \sigma)$ is the number of $F$-paths from $\tau$ to $\sigma$.
The \textit{discrete Morse complex} of $F$ on $K$, denoted $\mathcal{M}$, is defined as 
\[
0\rightarrow \ldots \xrightarrow[]{\partial_{p+2}} \mathcal{M}_{p+1} \xrightarrow[]{\partial_{p+1}} \mathcal{M}_{p} \xrightarrow[]{\partial_{p}} \mathcal{M}_{p-1} \xrightarrow[]{} \ldots \xrightarrow[]{\partial_1} \mathcal{M}_0\xrightarrow[]{\partial_0} 0.
\]

\begin{theorem}[{\cite[Theorem 5.4]{King_mod2}}]
\label{thm:standardHom}
    If $\mathcal{M}$ is the discrete Morse complex of an acyclic discrete vector field on a simplicial complex $K$, then 
    $$H_p(K; \Z/2\Z) \cong H_p(\mathcal{M}, \partial; \Z/2\Z) \coloneq \dfrac{\Ker(\partial_p; \Z/2\Z)}{\Im(\partial_{p+1}; \Z/2\Z)}.$$
\end{theorem}

To preserve persistent homology, some extra conditions will be required of $F$, as established by \cite{filtered}.
\begin{definition}[{\cite[Definition 4.1]{filtered}}]
    \label{def:filteredAcyclic}
    Given a filtration $\mathcal{F} = \{K^n \mid n = 1, 2, \dots N\}$ of a complex $K$, 
    a \textbf{filtered acyclic vector field} of $\mathcal{F}$ is a collection of acyclic discrete vector fields $F^n$, each on their respective complexes $K^n$, that preserves arrows and critical simplices.
    In particular, for all $\sigma \in K^n$, if $\sigma$ is a tail, a head, or critical in $F^n$, then it is respectively a tail, a head, or critical in $F^{n+1}$ on $K^{n+1}$, for all $n \in \{1, 2, \dots, N-1\}.$
    Additionally, if $\{\alpha \rightarrow \beta\}$ is in $F^{n+1}$, and $\alpha, \beta \in K^n$, then $\{\alpha \rightarrow \beta\}$ is in $F^{n}$.
\end{definition}

Given a filtered acyclic vector field, the acyclic vector field $F^n$ on $K^n$ induces its own Morse complex $\mathcal{M}^n$ for each $n$ given by 
\[
0\rightarrow \ldots \xrightarrow[]{\partial_{p+2}^n} \mathcal{M}^n_{p+1} \xrightarrow[]{\partial_{p+1}^n} \mathcal{M}^n_{p} \xrightarrow[]{\partial_{p}^n} \mathcal{M}^n_{p-1} \xrightarrow[]{} \ldots \xrightarrow[]{\partial_1^n} \mathcal{M}_0^n\xrightarrow[]{\partial^n_0} 0,
\]
where $\mathcal{M}^n_p$ is the span of $p$-simplices that are critical in $F^n$ 
and $\partial_{p}^n$ is the boundary map calculated via the number of $F^n$-paths between critical simplices.
This behavior is enough for the Morse complexes to compute the persistent homology of the filtration of $K$.

\begin{theorem}[{\cite[Proposition 4.2, Theorem 4.3]{filtered}}]
\label{thm:persHom}
The above Morse complexes $\{\mathcal{M}^n \mid n = 1, \dots, N\}$ form a filtration of the Morse complex $\mathcal{M} = \mathcal{M}^N$. 
In addition, for all $n, p,$ and $q$, 
$H^p_q(K^n) \cong H^p_q(\mathcal{M}^n)$.
\end{theorem}

\section{Colex Vector Fields}

In this section, we introduce an acyclic discrete vector field that is compatible with the colex order induced by a vertex function $f$.
This vector field is uniquely determined by $f$, and its restrictions to the frames of the lower star filtration form a filtered acyclic vector field, allowing us to use it to compute persistent homology.

\subsection{Candidate vertices and construction}

As before, let $K$ be a simplicial complex with vertex set $V$, and let $|V| = n$. 
Additionally, fix a given a bijection $$f \colon V \rightarrow \{1, 2, \dots, n\},$$  henceforth called an \emph{order function}. This function extends to the full complex as 
\[
\begin{matrix}
f \colon &K &\rightarrow &\{1, 2, \dots, n\}\\
& \sigma& \mapsto &  \max\{f(v) \mid v \in \sigma\}.
\end{matrix}
\]

Recall that $\sigma$ is represented as a set of vertices, so the notation $v \in \sigma$ implies $v \in V$.
We also define the \textit{appearance function} $\omega \colon K \rightarrow \{1,2,\dots, n\}$ by
\[
\omega(\sigma) \coloneq \min\{f(v) \mid v \in \sigma\}.
\]

\begin{definition}
\label{def:candidates}
For any $\sigma \in K$, define the \textbf{candidate list} of $\sigma$ as ${V_\sigma \coloneq \{v\in V \mid  \{v\} \cup \sigma \in K\}}$.
We then define the \textbf{candidate vertex} of $\sigma$ with respect to the order function $f$ as $c_f(\sigma) \coloneq \argmin_{v\in V_\sigma} f(v)$. 
\end{definition}
We remark that if the definition of $\overline \St$ from Definition \ref{def:starsAndLinks} is extended to simplices, then $V_\sigma = \overline \St(\sigma) \cap V$.


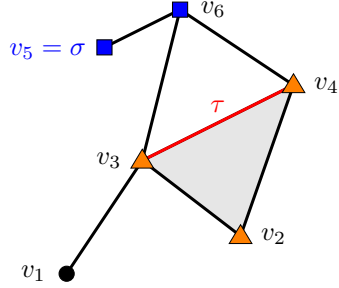
\begin{figure}[h!]
\begin{center}
    
\begin{tikzpicture}[scale=1.0]


\coordinate (v1) at (0,0.5);
\coordinate (v2) at (2.3,1);
\coordinate (v3) at (1,2);
\coordinate (v4) at (3,3);
\coordinate (v5) at (0.5,3.5);
\coordinate (v6) at (1.5,4);

\filldraw[line width=0.4mm, draw=black, fill=gray!20] (v2) -- (v3) -- (v4)-- cycle;
\draw[line width=0.4mm, black] (v1) -- (v3);
\draw[line width=0.4mm, black] (v6) -- (v4);
\draw[line width=0.4mm, black] (v3) -- (v6);
\draw[line width=0.4mm, black] (v5) -- (v6);
\draw[line width=0.4mm, red] (v4) -- (v3) node[midway, above] {$\tau$};

\draw[black, fill=black] (v1) circle (.1) node[left=1.5mm] {$v_1$};

\filldraw[line width=0.1mm, draw=black, fill=orange] (2.3,1.15) -- (2.45,0.9) -- (2.15,0.9) -- cycle;
\node[text=black, right=1.5mm of v2] {$v_2$};

\filldraw[line width=0.1mm, draw=black, fill=orange] (1,2.15) -- (1.15,1.9) -- (0.85,1.9) -- cycle;
\node[text=black, left=1.5mm of v3] {$v_3$};

\filldraw[line width=0.1mm, draw=black, fill=orange] (3,3.15) -- (3.15,2.9) -- (2.85,2.9) -- cycle;
\node[text=black, right=1.5mm of v4] {$v_4$};

\draw[black, fill=blue] (0.4,3.4) rectangle (0.6,3.6); 
\node[text=blue] at (-0.25,3.45) {$v_5 = \sigma$};

\draw[black, fill=blue] (1.4,3.9) rectangle (1.6,4.1);
\node[text=black, right=1.5mm of v6] {$v_6$};

\end{tikzpicture}
\captionsetup{width=.8\linewidth}
\caption{An example complex $K$, with $\sigma$ and $\tau$ labeled. The candidate list for $V_\sigma$, given by $\{v_5, v_6\}$,  is marked with blue squares, while $V_\tau = \{v_2, v_3, v_4\}$ is marked with orange triangles.}
\label{fig:candidates}

\end{center}
\end{figure}

Consider the example complex $K$ of Fig.~\ref{fig:candidates}, with the order function $f(v_i) = i$ for all vertices.
Consider $\sigma = \{v_5\}$ and $\tau = \{v_3, v_4\}$.
Since $v_5 v_6$ is a simplex in $K$, $v_6 \in V_\sigma$.
Additionally, $\{v_5\} \cup \{v_5\} = \{v_5\} \in K$, so $v_5$ is also in $V_\sigma$. 
No other vertices fit this criteria, so $V_\sigma = \{v_5, v_6\}$.
Using our order function $f$, we evaluate our candidate vertex as $c_f(\sigma) = v_5$.
Similarly, $V_\tau = \{v_2, v_3, v_4\}$, and evaluating $f$, we obtain $c_f(\tau) = v_2$.

\begin{lemma}
\label{lem:vSigmaStar}
For any $v \in V$ and $\sigma \in K$, we have $v \in V_\sigma$ iff $\sigma \in \overline \St(v)$.
\end{lemma}
\begin{proof}
If $v \in V_\sigma$, then $\{v\}\cup \sigma \in K$. 
Hence $\sigma \leq \{v\} \cup \sigma \in \St(v) \subset \overline \St(v)$. On the other hand, if $\sigma \in \overline \St(v)$, then $\sigma \leq \tau \in \St(v)$.
So $v$ and $\sigma$ are both faces of $\tau$, and hence so is $\{v\} \cup \sigma$.
Thus $\{v\} \cup \sigma$ is in $K$.
\end{proof}




\begin{definition}
\label{def:covf}
We define the \textbf{colex vector field induced by $f$}, denoted $F_f$, as follows: 
a pair $$\{\alpha^{(p)} \rightarrow \beta^{(p+1)}\}$$ is in $F_f$ if and only if $\beta = \{c_f(\alpha)\} \cup \alpha$ and $c_f(\alpha) \notin \alpha$.
\end{definition}

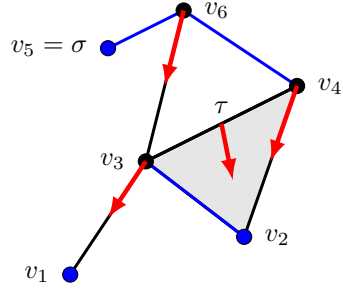
\begin{figure}[h!]
\begin{center}
\begin{tikzpicture}[scale=1.0]
\coordinate (v1) at (0,0.5);
\coordinate (v2) at (2.3,1);
\coordinate (v3) at (1,2);
\coordinate (v4) at (3,3);
\coordinate (v5) at (0.5,3.5);
\coordinate (v6) at (1.5,4);

\filldraw[line width=0.4mm, draw=black, fill=gray!20] (v2) -- (v3) -- (v4)-- cycle;
\draw[line width=0.4mm, black] (v1) -- (v3);
\draw[line width=0.4mm, blue] (v6) -- (v4);
\draw[line width=0.4mm, black] (v3) -- (v6);
\draw[line width=0.4mm, blue] (v5) -- (v6);
\draw[line width=0.4mm, black] (v4) -- (v3) node[midway, above] {$\tau$};
\draw[line width=0.4mm, blue] (v3) -- (v2);

\draw[black, fill=blue] (v1) circle (.1) node[left=1.5mm] {$v_1$};
\draw[black, fill=blue] (v2) circle (.1) node[right=1.5mm] {$v_2$};
\draw[black, fill=black] (v3) circle (.1) node[left=1.5mm] {$v_3$};
\draw[black, fill=black] (v4) circle (.1)
node[right=1.5mm] {$v_4$};
\draw[black, fill=blue] (v5) circle (.1) node[left=1.5mm, text=black] {$v_5 = \sigma$};
\draw[black, fill=black] (v6) circle (.1) node[right=1.5mm] {$v_6$};

\draw[-latex, red, line width = 0.6mm] (v3) -- ($(v3)!0.5!(v1)$);
\draw[-latex, red, line width = 0.6mm] (v4) -- ($(v2)!0.5!(v4)$);
\draw[-latex, red, line width = 0.6mm] (v6) -- ($(v3)!0.5!(v6)$);
\draw[-latex, red, line width = 0.6mm] ($(v3)!0.5!(v4)$) -- ($(v3)!0.5!(v4)!0.5!(v2)$);

\end{tikzpicture}
\captionsetup{width=.8\linewidth}
\caption{$F_f$ drawn on the same simplicial complex as in Figure \ref{fig:candidates}. Critical cells, which are not part of any arrow, are colored blue and listed for clarity: $\{v_1, v_2, v_2 v_3, v_5, v_4 v_6, v_5 v_6\}$.}
\label{fig:drawCovf}
\end{center}

\end{figure}

If there is an arrow of $F_f$ with tail $\alpha$, then the head of that arrow must be obtained by adding $c_f(\alpha)$; however, this is done only when adding that vertex yields a different simplex.
Continuing the same example of Figure \ref{fig:candidates}, in Figure \ref{fig:drawCovf} we draw every arrow of $F_f$ on $K$.
Observe that $\sigma$ is not the tail of any arrow, as $c_f(\sigma) = v_5 \in \sigma$.
On the other hand, $\tau$ is the tail of the arrow 
$v_3 v_4 \rightarrow v_2 v_3 v_4$, since $c_f(\tau) = v_2 \notin \tau$.

\begin{prop}
    \label{prop:covfProperties}
    Fix an order function $f$ on the vertices of a simplicial complex $K$, and extend $f$ to all of $K$ as before, and let $\omega$ be its appearance function.
    Suppose $\{\alpha \rightarrow \beta\} \in F_f$, where $ \beta =  \{c_f(\alpha)\} \cup \alpha $. Then:
    \begin{enumerate}[label=(\alph*)]
    
    \item $f(c_f(\alpha)) < f(v)$ for any vertex $v \in \alpha$, or equivalently, $c_f(\alpha) \colex v$.
    \item $f(c_f(\alpha)) < \omega(\alpha)$.
    \item $f(\alpha) = f(\beta)$.
    \item $\alpha$ is the unique codimension-1 face of $\beta$ such that $\omega(\alpha) > \omega(\beta)$.
    \item $\alpha \colex \beta$ are adjacent in colex order: that is, there exists no $\gamma \in K$ such that 
            $\alpha \colex \gamma \colex \beta$. 
    \item In the simplex-wise lower star filtration given by the colex order induced by $f$, the persistent homology contains the pair $(\alpha, \beta)$.
    \end{enumerate}   
\end{prop}
\begin{proof}$ $

    \begin{enumerate}[label=(\alph*)]
    \item \label{test}
    Fix $v \in \alpha$. It is also in $V_\alpha$ by definition. 
        Since $c_f(\alpha)$ is not in $\alpha$ and minimizes $f$ on $V_\alpha$, 
        we know $f(c_f(\alpha)) < f(v)$.
        Moreover, the colex order on $K$ is induced by the order function on the vertices, with $v_i \colex_f v_j$ if and only if $f(v_i) < f(c_j)$.
        Therefore $c_f(\alpha) \colex v$.
        
    \item Taking the minimum of the right side of (a) over all vertices in $\alpha$, 
    ${f(c_f(\alpha)) < \min_{v\in \alpha} f(v) = \omega(\alpha)}$.

    \item Taking the maximum of the right side of (a) over all vertices in $\alpha$, 
    ${f(c_f(\alpha)) < \max_{v\in \alpha} f(v) = f(\alpha)}$.
    In particular, $\argmax \{f(v) \mid  v \in \beta  = \{c_f(\alpha)\} \cup \alpha\}$ must be in $\alpha$.
    Hence \[\argmax \{f(v) \mid  v \in \beta \} = \argmax \{f(v) \mid  v \in \alpha \},\]
    and so $f(\beta) = f(\alpha).$

    \item By (a), $\argmin \{f(v) \mid  v \in \beta  = c_f(\alpha) \cup \alpha\} = c_f(\alpha)$.
    Hence $\omega(\beta) = f(c_f(\alpha)) < \omega(\alpha).$
    If $\tau < \beta$ is any other codimension-1 face, then $c_f(\alpha) \in \tau$, 
    so $\omega(\tau) \leq \omega(\beta)$.
    Therefore the described face is in fact unique.

    \item Suppose $\alpha \colex \gamma \colex \beta$. 
    Let $\alpha = \alpha_0 \alpha_1 \cdots \alpha_k$, $\beta = \beta_0 \beta_1 \cdots \beta_{k+1}$, and $\gamma = \gamma_0 \gamma_1 \cdots \gamma_\ell$.
    Recall that ${\beta = \{c_f(\alpha)\} \cup \alpha}$, so by (a), $\beta_0 = c_f(\alpha)$ and $\beta_i = \alpha_{i-1}$ for $i \in \{1, \dots, k+1\}$.
    Define $X$ as  ${\min\{i: \alpha_{k-i}\neq \gamma_{\ell-i}\}}$
    and $Y\coloneq \min\{i: \beta_{k+1-i}\neq \gamma_{\ell-i}\}$.
    If $X$ exists, then $X = Y$, since $\beta_{k+1-i} = \alpha_{k-i}$.
    On the other hand, if $X$ does not exist, then $\gamma = \cdots\gamma_{\ell-k-1}\alpha$.
    If $\gamma_{\ell-k-1} = \beta_0$, then $\gamma = \cdots \beta$, and $\beta \colex \gamma$.
    Therefore, ${\gamma_{\ell-k-1} \colex \beta_0 = c_f(\alpha)}$. 
    However, $\gamma_{\ell-k-1} \alpha \in K$, so by definition of $c_f$, ${c_f(\alpha) \colex \gamma_{\ell-k-1}}$, and we have a contradiction.

    \item By (d) and (e), $\alpha$ is the last codimension-1 face of $\beta$ to be added.
    Therefore, the addition of $\alpha$ creates the cycle $[\partial \beta]$ in homology.
    On the other hand, this will become 0 as soon as $\beta$ is added.
    Since $\beta$ is added in the very next step by (e), the cycle cannot die any earlier, and therefore these two simplices form a pair in persistent homology.

    
    \end{enumerate}

\end{proof}

As a consequence of the colex order part of (a), we will sometimes write $\beta = c_f(\alpha) \alpha$, 
or $\beta = v\alpha$ when it is clear that $v$ is the candidate vertex of $\alpha$.
We also note these implications do not necessarily go the other way.
In particular, for (d), $\alpha$ could be such a face for many different $\beta_i$.
For example, in Figure \ref{fig:drawCovf}, consider the two simplices $v_3 v_6$ and $v_4 v_6$.
For both simplices, $\{v_6\}$ is the unique codimension-1 face with a higher value of $\omega$.
However, $\{v_6 \rightarrow v_3 v_6\}$ is in $F_f$, while $\{v_6 \rightarrow v_4 v_6\}$ is not.
We remark that in general, when $\alpha$ is such a face for many different $\beta_i$, the definition of $c_f$ will pair $\alpha$ with the coface $\beta_i$ that minimizes $\omega(\beta_i)$.

\begin{prop}
    $F_f$ is a discrete vector field; that is, no simplex is in more than one pair.
\end{prop}
\begin{proof}
    Since $\{c_f(\alpha)\} \cup \alpha$ is defined uniquely for each $\alpha$, no $\alpha$ can be the tail of two arrows.
    Similarly, by Prop.~\ref{prop:covfProperties}(d), a simplex $\beta$ can only be the head of the arrow $\{\alpha \rightarrow \beta\}$ where $\alpha$ is the unique codimension-1 face of $\beta$ such that $\omega(\alpha) > \omega(\beta)$.
    Therefore, if some $\beta$ is in two arrows, it must be the head of one and the tail of another.
    
    Suppose now that $\{\alpha \rightarrow \beta = v \alpha\}$ and $\{\beta \rightarrow \gamma =  u\beta\}$ are two arrows in $F_f$ with $u, v \in V$. 
    Since $\gamma = u\beta$ and $v \in \beta$, by Prop.~\ref{prop:covfProperties}(a)
    we must have $f(u) < f(v)$.
    Moreover, the simplex $u\alpha$ must be in $K$, since it is a face of $\gamma$. 
    In particular, $u \in V_\alpha$.
    Hence $f(c_f(\alpha)) \leq f(u) < f(v)$, so $c_f(\alpha) \neq v$, and we have a contradiction.    
\end{proof}

\begin{prop}
    \label{prop:acyclic}
     There are no nontrivial closed $F_f$-paths.
\end{prop}
\begin{proof}
    Suppose we have a nontrivial $F_f$-path $\alpha_0 \rightarrow \beta_0 > \alpha_1 \rightarrow \dots >\alpha_{r+1}$.
    For each ${i \in \{0, 1, \dots, r\}}$, $\omega(\alpha_i) > \omega(\beta_i)$ by Prop. \ref{prop:covfProperties} (d).
    Since $\alpha_{i+1}$ is a different codimension-1 face, we must have $\omega(\alpha_{i+1}) = \omega(\beta_i)$.
    Hence 
    \[ \omega(\alpha_{r+1}) < \omega(\alpha_{r}) < \dots < \omega(\alpha_1) < \omega(\alpha_0)\]
    and in particular, $\alpha_{r+1} \neq \alpha_0$.    
\end{proof}

Hence $F_f$ is an acyclic discrete vector field, and by Theorem \ref{thm:standardHom}, we may use it to compute the homology of $K$. 
We end this section by showing that it computes the persistent homology of the lower star filtration.
For this, we must refine the lower star filtration to its simplex-wise subfiltration as follows.
Let ${\{\sigma_1, \sigma_2, \dots, \sigma_N\}}$ be the list of all simplices of $K$ sorted in colex order with respect to $f$, and for each $i \in \{0, \dots, N\}$, define $K^i \coloneq \{\sigma_j \mid j \leq i\}$.
Then $\tilde{\mathcal{K}} \coloneq \{\emptyset = K^0 \subset K^1 \subset \cdots \subset K^N = K\}$ is the\textit{simplex-wise lower star filtration} of $K$ induced by the colex order associated with $f$.

By Prop.~\ref{prop:covfProperties} (e), every arrow of $F_f$ is of the form $\{\sigma_i \rightarrow \sigma_{i+1}\}$.
We wish to construct a filtration for which the restrictions of $F_f$ form a filtered acyclic vector field as in Definition \ref{def:filteredAcyclic}.
If the head and tail of an arrow first appear in different frames of our filtration, this definition will not be met. So, consider
\begin{equation}
  \label{eq:tails}
T \coloneq \{i \in \{1, \dots, N-1\} \mid \{\sigma_i \rightarrow \sigma_{i+1}\} \in F_f\},
\end{equation}
the set of indices of tails of arrows.
We use these to specify the slightly less refined filtration $\mathcal{K}'$, obtained by ignoring each of the $K^i$ for $i \in T$.
In particular, letting $J \coloneq \{1, \dots, N\} \setminus T$ and indexing it as $\{j_1 < j_2 < \dots <j_{N-|T|}\}$, 
we may use it as the indices for our filtration
\begin{equation}
  \label{eq:covfFiltration}
\mathcal{K}' \coloneq \{\emptyset = K^0 \subset K^{j_1} \subset K^{j_2} \subset \cdots \subset K^{j_{N-|T|}} = K^N = K\}.
\end{equation}

For each $j \in J$, let $F^{j}_f$ be the restriction of $F_f$ to $K^j \in \mathcal{K}'$; that is, $\{\alpha \rightarrow \beta\}$ is in $F^j_f$ if and only if $\alpha, \beta \in K^j$ and $\{\alpha \rightarrow \beta\} \in F_f$.

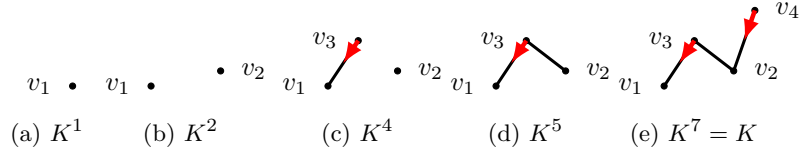
\begin{figure}[h!]
\centering
\begin{subfigure}{.1\textwidth}
  \centering
  \begin{tikzpicture}[scale=0.4]
\coordinate (v1) at (0,0.5);

\draw[black, fill=black] (v1) circle (.1) node[left=1.5mm] {$v_1$};

\end{tikzpicture}
  \caption{$K^1$}
\end{subfigure}%
\begin{subfigure}{.17\textwidth}
  \centering

  \begin{tikzpicture}[scale=0.4]
\coordinate (v1) at (0,0.5);
\coordinate (v2) at (2.3,1);

\draw[black, fill=black] (v1) circle (.1) node[left=1.5mm] {$v_1$};
\draw[black, fill=black] (v2) circle (.1) node[right=1.5mm] {$v_2$};

\end{tikzpicture}
  \caption{$K^2$}
\end{subfigure}
\begin{subfigure}{.17\textwidth}
  \centering

  \begin{tikzpicture}[scale=0.4]
\coordinate (v1) at (0,0.5);
\coordinate (v2) at (2.3,1);
\coordinate (v3) at (1,2);

\draw[line width=0.4mm, black] (v1) -- (v3);

\draw[black, fill=black] (v1) circle (.1) node[left=1.5mm] {$v_1$};
\draw[black, fill=black] (v2) circle (.1) node[right=1.5mm] {$v_2$};
\draw[black, fill=black] (v3) circle (.1) node[left=1.5mm] {$v_3$};

\draw[-latex, red, line width = 0.6mm] (v3) -- ($(v3)!0.5!(v1)$);

\end{tikzpicture}
    \caption{$K^4$}
\end{subfigure}%
\begin{subfigure}{.17\textwidth}
  \centering

  \begin{tikzpicture}[scale=0.4]
\coordinate (v1) at (0,0.5);
\coordinate (v2) at (2.3,1);
\coordinate (v3) at (1,2);

\draw[line width=0.4mm, black] (v1) -- (v3);
\draw[line width=0.4mm, black] (v2) -- (v3);

\draw[black, fill=black] (v1) circle (.1) node[left=1.5mm] {$v_1$};
\draw[black, fill=black] (v2) circle (.1) node[right=1.5mm] {$v_2$};
\draw[black, fill=black] (v3) circle (.1) node[left=1.5mm] {$v_3$};

\draw[-latex, red, line width = 0.6mm] (v3) -- ($(v3)!0.5!(v1)$);

\end{tikzpicture}
 \caption{$K^5$}
\end{subfigure}%
\begin{subfigure}{.17\textwidth}
  \centering
  \begin{tikzpicture}[scale=0.4]
\coordinate (v1) at (0,0.5);
\coordinate (v2) at (2.3,1);
\coordinate (v3) at (1,2);
\coordinate (v4) at (3,3);

%
\draw[line width=0.4mm, black] (v1) -- (v3);
\draw[line width=0.4mm, black] (v2) -- (v3);
\draw[line width=0.4mm, black] (v2) -- (v4);

\draw[black, fill=black] (v1) circle (.1) node[left=1.5mm] {$v_1$};
\draw[black, fill=black] (v2) circle (.1) node[right=1.5mm] {$v_2$};
\draw[black, fill=black] (v3) circle (.1) node[left=1.5mm] {$v_3$};
\draw[black, fill=black] (v4) circle (.1) node[right=1.5mm] {$v_4$};

\draw[-latex, red, line width = 0.6mm] (v3) -- ($(v3)!0.5!(v1)$);
\draw[-latex, red, line width = 0.6mm] (v4) -- ($(v2)!0.5!(v4)$);

\end{tikzpicture}
\caption{$K^7 = K$}
\end{subfigure}
\captionsetup{width=.8\linewidth}
\caption{A depiction of $\mathcal{K'}$ of Equation \ref{eq:covfFiltration}, with $F^j_f$ overlaid. We use $K^j$ to refer to the complex created by the first $j$ simplices of $K$ in colex order, with $K^0$ omitted. Each step adds the next simplex in colex order, or if that simplex is part of an arrow, it adds both simplices of the arrow instead.}
\label{fig:filtered}
\end{figure}

We provide an example of this in Figure \ref{fig:filtered}, where $K = K^8$ and the order function is given by $f(v_i) = i$ for all vertices.
Here, \[\mathcal{K}' = \{\emptyset = K^0 \subset K^1 \subset K^2 \subset K^4 \subset K^5 \subset K^7 = K\}.\]
Since $\{v_3 \rightarrow v_1 v_3\} \in F_f$, including $K^3 = \{ \{v_1\}, \{v_2\}, \{v_3\}\}$ in the filtration would add the two simplices of the arrow at different times, preventing these restrictions from being a filtered acyclic vector field.

\begin{prop}
    For a fixed order function $f$ on a simplicial complex $K$, let $J$ be the non-tail indices of $F_f$ as above, and let $F^{j}_f$ be the restriction of $F_f$ to $K^j \in \mathcal{K}'$ for each $j$.
    Define $\mathcal{F} \coloneq \{ F^{j}_f \mid j \in J\}$.
    Then $\mathcal{F}$ is a filtered acyclic vector field of $\mathcal{K}'.$
\end{prop}
\begin{proof}
    First we show that each $F^{j}_f$ is an acyclic discrete vector field.
    Since it is a restriction of $F_f$, no simplex is in more than one arrow.
    Now suppose that $F^{j}_f$ has a non-trivial closed path
    \[\alpha_0 \rightarrow \beta_0 > \alpha_1 \rightarrow \dots >\alpha_{r+1} = \alpha_0.\]
    Since each arrow is in $F_f$, and each face relation still holds in $K$, this would also form a nontrivial closed path in $F_f$, which contradicts Prop.~\ref{prop:acyclic}.

    Now, fix $\sigma \in K$, 
    and let $j$ be such that $\sigma \notin K^{j-1}$ but $\sigma \in K^j$.
    Let $\ell \geq j$.
    If $\sigma$ is critical in $F_f$, then since $F^\ell_f$ is a restriction,
    it will be critical in $F^\ell_f$.
    If $\sigma$ is a tail of the arrow $\sigma \rightarrow \tau \in F_f$, 
    then by construction, we also have $\tau \notin K^{j-1}$, $\tau \in K^j$.
    As a restriction, the arrow $\sigma \rightarrow \tau$ is also present in $F^j_f$, as well as $F^\ell_f$.
    The same argument holds if $\sigma$ is the tail of an arrow.
    Thus $\mathcal{F}$ fits all requirements of Definition \ref{def:filteredAcyclic}.
\end{proof}

By Theorem \ref{thm:persHom}, the Morse complex associated with $F_f$ can be used to compute the persistent homology of $\mathcal{K}'$.
We end by showing that the interval decomposition of $H_*(\mathcal{K}')$ matches that of the simplex-wise lower star filtration. 
For this we require a technical lemma from \cite{oudot}.
We preemptively mention that in the lemma, the function $t$ on simplices is induced by their first appearance time in a fixed filtration; this is not the same as the order function $f$.

\begin{lemma}{(\cite{oudot}, Lemma 2.12)}
    Let $\mathcal{K}$ be a filtration ${\emptyset \subseteq K_{t_1} \subseteq \dots \subseteq K_{t_n} = K}$, 
    and define $t(\sigma) = \min\{t_i \mid \sigma \in K_{t_i}\}$.
    Let $K = \{\sigma_1, \sigma_2, \dots, \sigma_m\}$ be a total order on $K$ that is compatible with $t$ and the incidence relations between level sets; that is, $i < j$ whenever $t(\sigma_i) < t(\sigma_j)$ or $\sigma_i < \sigma_j$.
    Define $\tilde{\mathcal{K}}$ as the filtration \[\emptyset \subseteq \{ \sigma_1 \} \subseteq \{ \sigma_1, \sigma_2 \} \subseteq \dots \subseteq \{ \sigma_1, \dots, \sigma_m \} = K.\]

    Then there is a partial matching between the interval decompositions of $H_*(\mathcal{K})$ and $H_*(\tilde{\mathcal{K}})$ such that:
    \begin{itemize}
        \item summands $\mathbb{I}[i, +\infty)$ of $\mathsf{H}_p(\tilde{\mathcal{K}})$ are matched with summands $\mathbb{I}[t(\sigma_i), +\infty)$ of $\mathsf{H}_p(\mathcal{K})$, and vice-versa,

        \item summands $\mathbb{I}[i, j)$ of $\mathsf{H}_p(\tilde{\mathcal{K}})$ where $t(\sigma_i) < t(\sigma_j)$ are matched with summands $\mathbb{I}[t(\sigma_i), t(\sigma_j))$ of $\mathsf{H}_p(\mathcal{K})$, and vice-versa,

        \item all other summands of $\mathsf{H}_p(\mathcal{K})$, i.e. summands $\mathbb{I}[i, j)$ where $t(\sigma_i) = t(\sigma_j)$, are unmatched.
    \end{itemize}
\end{lemma}

\begin{theorem}
    The interval decomposition of $\mathsf{H}_*(\tilde{\mathcal{K}})$ is exactly composed of the intervals of the decomposition of $\mathsf{H}_*(\mathcal{K}')$ combined with 
    $\{\mathbb{I}[i, i+1) \mid i \in T\}$, where $T$ is the set of indices of tails of $F_f$ as in Equation \ref{eq:tails}.
\end{theorem}
\begin{proof}
    In the language of the above lemma, 
    $t(\sigma_i) = i+1$ if $i \in T$, and $t(\sigma_i) = i$ otherwise.
    So a summand $\mathbb{I}[i, j)$ of $\mathsf{H}_*(\tilde{\mathcal{K}})$ where $t(\sigma_i) = t(\sigma_j)$ means that $j = i+1$ and $i \in T$.
    Therefore, in the partial matching, every unmatched summand of $\mathsf{H}_p(\tilde{\mathcal{K}})$ is in fact $\mathbb{I}[i, i+1)$, corresponds to the arrow $\{\sigma_i \rightarrow \sigma_{i+1}\} \in F_f$.
    Moreover, by Prop.~\ref{prop:covfProperties} (f), every arrow of $F_f$ induces such a summand.
\end{proof}

In summary, the Morse complex of $F_f$ can be used to compute the persistent homology of the lower star filtration of $K$, with the arrows of $F_f$ recovering the lost pairs.

\section{Adjacent Vertex Swaps and Implementation}

In this section, we give the updates required for the discrete vector field when two vertices in an order function are swapped.

\subsection{Efficient Colex Vector Field Updates}
\label{subsec:fastUpdates}

Given a fixed simplicial complex $K$ and two arbitrary order functions $f$ and $g$ on $K$, one expects to find no similarities between the colex vector fields $F_f$ and $F_g$.
However, if $f$ and $g$ have similar behavior, one might expect similar behavior between $F_f$ and $F_g$.
In this vein, we describe how to update $F_f$ efficiently when the underlying order function $f$ changes slightly.

\begin{definition}
    Given a vertex set $V$ and an order function $f \colon V \rightarrow \{1, 2, \dots, |V|\}$, 
    we say that two vertices $x, y \in V$ are order-adjacent if $|f(x) - f(y)| = 1$.
\end{definition}

We consider the case where the order function changes via a swap of two order-adjacent vertices as follows.
Let $x$ and $y$ be order-adjacent vertices where $x$ comes first in the order of $f$: that is, $f(y) = f(x) + 1$.
Suppose $f$ is the first order function, and $f'$ is the order function after the swap of $x$ and $y$: in particular, for any $v \in V$,

\begin{equation}
\label{eq:fPrime}
f'(v) \coloneq \begin{cases} 
      f(x) & v = y \\
      f(y) & v = x \\
      f(v) & v \notin \{x, y\}.
   \end{cases}
\end{equation}
We also compute the difference between the two orders:
\[
(f'-f)(v) \coloneq \begin{cases} 
      -1 & v = y \\
      1 & v = x \\
      0 & v \notin \{x, y\}.
   \end{cases}
\]

\begin{lemma}
\label{lem:swappingLemma}
For any $a, b \in V$, $f(a) < f(b)$ and $f'(a) > f'(b)$ iff $a = x, b = y$.
\end{lemma}
\begin{proof}
For the forward direction, suppose on the contrary that $a \notin \{x, y\}.$
Then $f(a) = f'(a)$, so we have ${f'(b) < f(a) < f(b)}$. 
Thus $0 < f(a) - f'(b) < f(b) - f'(b) \leq 1$.
However, no integer value for $f(a) - f'(b)$ exists to satisfy this, which is a contradiction, so $a \in \{x, y\}$. 
A similar argument shows $b \in \{x, y\}$, and a simple case analysis finishes the proof. The reverse direction is easily checked by computation.
\end{proof}

Recall that for any $\sigma \in K$, its candidate list $V_\sigma$ (Definition \ref{def:candidates}) does not depend on the order function, only the underlying complex $K$. 
On the other hand, for a fixed ${\sigma \in K}$, the candidate vertices ${c_f(\sigma) \coloneq \argmin_{v\in V_\sigma} f(v)}$ and ${c_{f'}(\sigma) \coloneq \argmin_{v\in V_\sigma} f'(v)}$ may differ.
Recall that the colex vector field is uniquely determined by the candidate vertex of each simplex.
In order to determine what changes from $F_f$ to $F_{f'}$, we will first determine which candidate vertices will change.

\begin{lemma}
\label{lem:vSigmaSwap}
Let $x$ and $y$ be order-adjacent vertices of $f$ with $f(y) = f(x) + 1$, and let $f'$ be the order function after their swap, as in Eq.~\ref{eq:fPrime}.
Then $c_{f}(\sigma) \neq c_{f'}(\sigma)$ iff $c_{f}(\sigma) = x$ and $c_{f'}(\sigma) = y$.
In particular, if $c_{f}(\sigma) \neq c_{f'}(\sigma)$, then $x, y \in V_\sigma$.

\end{lemma}
\begin{proof}
If $c_{f}(\sigma) \neq c_{f'}(\sigma)$, then we have $f(c_{f}(\sigma)) < f(c_{f'}(\sigma))$ and $f'(c_{f}(\sigma)) > f'(c_{f'}(\sigma))$.
Apply Lemma \ref{lem:swappingLemma}.
For the final statement, recall that $c_f$ and $c_{f'}$ are chosen from $V_\sigma$ (Definition \ref{def:candidates}). 
\end{proof}


Therefore, if we want to determine which candidate vertices will change, we are only interested in $\sigma$ such that $x, y \in V_\sigma$.
By Lemma \ref{lem:vSigmaStar}, this is equivalent to ${\sigma \in \overline \St(x) \cap \overline \St(y)}$.
Moreover, if $c_{f}(\sigma) = x$, then this change in candidate vertices is forced.

\begin{lemma}
\label{lem:forced}
Let $x$ and $y$ be order-adjacent vertices of $f$ with $f(y) = f(x) + 1$, and let $f'$ be the order function after their swap, as in Eq.~\ref{eq:fPrime}.
Suppose $\sigma \in K$ with $x, y \in V_\sigma$. 
Then $c_{f}(\sigma) = x$ iff $c_{f'}(\sigma) = y$.
\end{lemma}
\begin{proof}
If $c_{f}(\sigma) = x$, then the second-lowest value of $f$ is achieved by $y$.
Therefore it must achieve the lowest value of $f'$.
The reverse direction is the same.    
\end{proof}










At this point, one wishes to find all $\sigma \in \overline \St(x) \cap \overline \St(y)$ with $c_f(\sigma) = x$, and then assign $c_{f'}(\sigma) = y$.
In theory, storing $c_f(\sigma)$ for every $\sigma \in K$ and updating every single one after each vertex swap is possible, but this information is not directly encoded in $F_f$.
For example, if $\sigma$ is not part of any arrow of $F_f$, we  know that $c_f(\sigma) \in \sigma$, but we do not know its identity unless we examined all vertices of $\sigma$.

The remainder of this subsection instead performs a case analysis to speed up this process.
Fix $\sigma \in \overline \St(x) \cap \overline \St(y).$
Since we are updating $F_f$, we can assume that we already know the arrow $\sigma \rightarrow \tau$ (if it exists) from the function $f$. 
This provides us with three potential cases:

\begin{enumerate}[label=(\alph*)]
    \item $\sigma$ is not the tail of any arrow in $F_f$, or equivalently, $c_f(\sigma) \in \sigma$.
    
    \item $\tau = x \sigma$, or equivalently, $c_f(\sigma) = x$ and $x \notin \sigma$.
    
    \item $\tau = v \sigma$  with $v \neq x$, or equivalently, $c_f(\sigma) \neq x$, $c_f(\sigma) \notin \sigma$.

\end{enumerate}

In Cases (b) and (c), $\tau$ identifies $c_f(\sigma)$, so the above lemmas describe exactly what must happen.
In Case (b), $c_{f'}(\sigma) = y$ by Lemma \ref{lem:forced}.
So, the arrow $\{\sigma \rightarrow x\sigma\}$ must get deleted from $F_f$.
Then, if $y \notin \sigma$, we include the arrow $\{\sigma \rightarrow y \sigma\}$ in $F_{f'}$. 
In Case (c), no change to the arrow will occur due to Lemma \ref{lem:vSigmaSwap}.

One possible way to handle Case (a) is to simply compute $c_f(\sigma)$, determine if it is $x$, and follow the above methods.
However, we note an optimization that may arise from finding $\overline \St(x) \cap \overline \St(y)$.
In particular, since Lemma \ref{lem:closedStarPartition} implies that 
\[
\overline \St(x) \cap \overline \St(y) = 
(\St(x) \sqcup \Lk(x)) \cap (\St(y) \sqcup \Lk(y)),
\]
when $\sigma \in \St(x) \cap \overline \St(y)$ is found, we may keep track of whether it was in $\St(x)$ or $\Lk(x)$, and similarly whether it was in $\St(y)$ or $\Lk(y)$.

\begin{lemma}
    \label{lem:caseA}
    In Case (a), the arrow $\{\sigma \rightarrow y \sigma\}$ appears in $F_{f'}$ if and only if ${\sigma \in \St(x) \cap \Lk(y)}$ and $\omega(\sigma) = f(x)$.
\end{lemma}
\begin{proof}
Suppose $\{\sigma \rightarrow y \sigma\}$ is in $F_{f'}$. 
Then by Definition \ref{def:covf}, $c_{f'}(\sigma) = y$, and $y \notin \sigma$.
Since $\sigma \in \overline \St(x) \cap \overline \St(y)$,
Lemma \ref{lem:forced} implies that $c_f(\sigma) = x$.
Since we are in Case (a), the lack of arrow with tail $\sigma$ in $F_f$ implies $x \in \sigma$.
In particular, the earliest vertex of $\sigma$ is $x$, and therefore $\omega(\sigma) = f(x)$.
Additionally, applying Lemma \ref{lem:closedStarPartition} to both $x$ and $y$ yields $\sigma \in \St(x) \cap \Lk(y)$.

For the converse, the statement $\sigma \in \St(x) \cap \Lk(y)$ implies that $x \in \sigma$ and $y \notin \sigma$.
Since $\omega(\sigma) = f(x)$, $c_f(\sigma) = x$. 
By Lemma \ref{lem:forced}, $c_{f'}(\sigma) = y \notin \sigma$, so the arrow $\{\sigma \rightarrow y \sigma\}$ appears in $F_{f'}$.
\end{proof}
In particular, if $\sigma \in \overline \St(x) \cap \overline \St(y)$ is located outside of $\St(x) \cap \Lk(y)$, no additional computation is needed, as no arrow will appear.
However, if $\sigma \in \St(x) \cap \Lk(y)$, it is necessary to spend extra time evaluating $\omega(\sigma)$.

The entire process is summarized in Algorithm \ref{alg:covf_update}.
It takes inputs of the simplicial complex $K$, an order function $f$, two order-adjacent vertices $x, y$ with $f(y) = f(x) + 1$, and the colex vector field $F_f$.
The algorithm outputs the colex vector field $F_{f'}$ where $f'$ is the order function after the swap of $x$ and $y$.

\begin{algorithm}[h!]
\caption{Colex vector field update for order-adjacent vertex swaps}\label{alg:covf_update}
\begin{algorithmic}[1]
    \State $F = F_f$;
    \ForAll{$\sigma \in \overline \St(x) \cap \overline \St(y)$}
        \If{$\{\sigma \rightarrow \tau\} \in F$ for some $\tau \in K$}
            \If{$\tau = x\sigma$}
                \State remove arrow $\{\sigma \rightarrow \tau\}$ from $F$
                \If{$y \notin \sigma$}
                    \State add arrow $\{\sigma \rightarrow y\sigma\}$ to $F$
                \EndIf
            \EndIf
        \ElsIf{$\sigma \in \St(x) \cap \Lk(y)$}
            \State $\omega = \min\{f(v) \mid v \in \sigma\}$
            \If{$\omega = f(x)$}
                \State add arrow $\{\sigma \rightarrow y\sigma\}$ to $F$
            \EndIf
        \EndIf
    \EndFor
\end{algorithmic}
\end{algorithm}

\subsection{Implementation of $F_f$, Structure of Path Matrices and Boundary Matrices}
\label{subsec:matrices}

For complexity analysis, we assume that locating a simplex $\sigma \in K$ occurs in constant time.
By storing information about arrows including $\sigma$ in the same location, 
we can also look up whether $\sigma$ is a head, tail, or critical in constant time.
Finally, we assume that we can evaluate $f$ on vertices in constant time.

Given a simplicial complex and an order function $f$ on its vertices, we must construct $F_f$, identify its critical cells, and compute the number of paths between pairs of critical cells.
We wish for this information to be compatible with Section \ref{subsec:fastUpdates}, in the sense that recomputing it after a vertex swap is efficient.
After a vertex swap, the set of critical simplices may change wildly, and we will have to compute the number of paths between them.
In our implementation, we instead record the number of gradient paths between \textit{any} two simplices of adjacent dimensions, storing them in ``path matrices".
These matrices do not record the actual paths taken, only the total amount.
We describe the process used to construct $F_f$ and the relevant path matrices first.

Let $V = \{ v_1, v_2, \dots, v_n\}$, where $f(v_i) = i$, taking $O(n \log n)$ time to sort them if needed.
Additionally, let $m$ be the number of simplices of $K$, $m_p$ be the number of $p$-dimensional simplices of $K$, and let $d \coloneq \max_{\sigma \in K} \dim(\sigma)$ be the dimension of $K$.
To create $F_f$, we iterate through all simplices $\sigma$ in $K$.
For a given $\sigma$, we calculate $c_f(\sigma)$: if we go through $V$ in order, then $c_f(\sigma)$ is the first vertex $v_i$ for which $\{v_i\} \cup \sigma$ is in $K$.
We then determine if $c_f(\sigma)$ is in $\sigma$, which takes at most $O(d)$ time.
If it is, then $\sigma$ is not the tail of an arrow, and we may move to the next simplex.
Otherwise, we include the arrow $\{\sigma \rightarrow v_i \sigma\}$ into $F_f$.
We store that $\sigma$ corresponds to a head of $v_i \sigma$, and  we store that $v_i \sigma$ corresponds to a tail of $\sigma$.
This process takes $O(dn)$ time for each $\sigma$, for a total of $O(dmn)$.

Recall that $\Gamma(\tau, \sigma)$ is the number of $F$-paths from $\tau$ to $\sigma$.
For each dimension $p$, we will keep a path matrix that stores $\Gamma(\tau^{(p+1)}, \sigma^{(p)})$ for any pair of simplices of those dimensions.
We construct matrices $\PP =\{P_0, P_1, \dots, P_d\}$, where each $P_p$ has rows indexed by the $p$-dimensional simplices $\{\sigma_i^{(p)}\}$ and columns indexed by $\{\tau_j^{(p+1)}\}$.
We wish for the entry corresponding to the row of $\sigma$ and the column of $\tau$, denoted $P_p[\sigma, \tau]$ to have value $\Gamma(\tau,\sigma)$.
To compute this, we will run through every arrow one at a time, adjusting the relevant values of $P_p$ until all of $F_f$ has been accounted for.
We think of this as starting with an empty discrete vector field and adding one arrow at a time, changing the vector field slightly, until we have reached $F_f$.

Before we have added any arrows, the only possible $F_f$-paths are the face relations $\beta > \alpha$, so our path matrices are simply the standard boundary matrices of $K$. 
We note that the rows and columns of $\PP$ do not need to be in colex order; this will be done on the boundary matrices of the Morse complex.

As we change our vector field in the future, we will see that updating $\PP$ relies on its acyclic nature.
We wish to ensure that after adding or deleting arrows, it remains an acyclic discrete vector field, leading to the following definition.

\begin{definition}
    Suppose that we wish to add an arrow $\alpha \rightarrow \beta$ to an acyclic discrete vector field $F$, changing it from $F$ to $F'$.
    We say that this occurs \textbf{naturally} if ${F' = F \sqcup \{\alpha \rightarrow \beta\}}$ and $F'$ is acyclic.   

    A deletion of $\alpha \rightarrow \beta$, changing $F'$ to $F$, occurs \textbf{naturally} if $F'$ can be obtained by naturally adding $\alpha \rightarrow \beta$ to $F$.
    We note that the acyclic condition is implicit for deletions: if $F'$ is acyclic, $F$ is too.
    

\end{definition}

\begin{prop}
\label{prop:countingAddition}
    Suppose the arrow $\alpha^{(p)} \rightarrow \beta^{(p+1)}$ is added to $F_f$ naturally.
    Then no new $F_f$-paths ending at $\alpha$ are created, and no new $F_f$-paths starting at $\beta$ are created.
    Additionally, for any $\tau^{(p+1)} \neq \beta$ and $\sigma^{(p)} \neq \alpha$,
    the number of $F_f$-paths from $\tau$ to $\sigma$ increases by $\Gamma(\tau, \alpha) \Gamma(\beta, \sigma)$.
\end{prop}
\begin{proof}
    We begin by noting that no $F_f$-paths are destroyed.
    Some paths may have used the relation $\beta > \alpha$, and they are still free to do so: the arrow $\alpha \rightarrow \beta$ does not prevent this face relation from being used.

    If a new path is constructed, it must use the arrow $\alpha^{(p)} \rightarrow \beta^{(p)}$.
    Let $\tau^{(p+1)}$ and $\sigma^{(p)}$ be arbitrary simplices of those dimensions.
    Then any new path from $\tau$ to $\sigma$ is of the form
    \[
    \tau > \cdots > \alpha \rightarrow \beta > \gamma \rightarrow \dots > \sigma,
    \]
    where $\gamma \neq \alpha$ by Definition \ref{def:FPath}.

    Any new paths that end at $\alpha$ must include $\alpha \rightarrow \beta > \cdots > \alpha$, but due to the acyclic nature of $F_f$, this cannot exist.
    Similarly, any new paths from $\beta$ would include $\beta > \cdots > \alpha \rightarrow \beta$, so they cannot exist either.
    Hence $\Gamma(\cdot, \alpha)$ and $\Gamma(\beta, \cdot)$ remain constant.
    
    Now assume $\tau \neq \beta$ and $\sigma \neq \alpha$.
    In this case, the new path uniquely corresponds to the concatenation of an existing path from $\tau$ to $\alpha$ and an existing path from $\beta$ to $\sigma$.
    Moreover, any two such paths can be concatenated along $\alpha \rightarrow \beta$, forming a new path.
    Suppose ${\tau > \cdots \rightarrow \theta > \alpha}$ and $\beta > \phi \rightarrow \dots > \sigma$ are two such paths.
    The only way their concatenation might fail to be an $F_f$-path is if 
    $\theta = \beta$ or $\phi = \alpha$.
    However, since $\alpha \rightarrow \beta$ was added naturally, neither $\alpha$ nor $\beta$ were part of an arrow before, so this is impossible.
    Hence the total number of new paths is the product of these, which is exactly $\Gamma(\tau, \alpha) \Gamma(\beta, \sigma)$.
\end{proof}

We mention that $\alpha^{(p)} \rightarrow \beta^{(p+1)}$ only affects $P_p$, as using that arrow immediately implies which dimensions a new $F_f$-path must alternate between.
To track this efficiently in $P_p$, we define the column vector 
\[
\Vec{w} \coloneq  \Vec{P_p}[\cdot, \beta] - 
P_p[\alpha, \beta]\Vec{e_\alpha},
\]
where $\Vec{e_\alpha}$ is the standard basis vector with a 1 in the position of the row corresponding to $\alpha$; in other words, $\Vec{w}$ is simply the column of $\beta$ with a 0 in the row corresponding to $\alpha$.
Then, to each column $\Vec{P_p}[\cdot, \tau]$ with $\tau \neq \beta$, we add $P_p[\alpha, \tau]\Vec{w}$.
This causes each $P_p[\sigma, \tau]$ to increase by $P_p[\alpha, \tau] P_p[\sigma, \beta]$, unless $\sigma = \alpha$ or $\tau = \beta$.

Adding $\alpha^{(p)} \rightarrow \beta^{(p+1)}$ incurs a cost of $O(m_p m_{p+1})$ to update $P_p$.
The process of adding all arrows is therefore $O(n M^2)$, where $M = \max_p\{ m_p\}$.
This does raise the question about the order in which arrows must be added.
Thankfully, any order will suffice.

\begin{lemma}
    If $F$ is a subfield of an acyclic discrete vector field $F'$, 
    then adding the arrows of $F' \setminus F$ to $F$ in any order will be done naturally for each arrow.
\end{lemma}
\begin{proof}
    Since subfields of acyclic discrete vector fields are acyclic, 
    $F \cup \{\alpha \rightarrow \beta\}$ is acyclic for any arrow 
    $\{\alpha \rightarrow \beta\} \in F' \setminus F$.
    Therefore, the first addition of an arrow happens naturally.
    Repeat this argument for each arrow in the chosen order.
\end{proof}

Once all of the necessary arrows are added, to create the boundary matrix $\partial_{p} \colon \mathcal{M}_{(p+1)} \rightarrow \mathcal{M}_{(p)}$, we use $F_f$ to determine all critical cells of dimension $p$ and $p+1$. 
We find each simplex of those dimensions and identify the ones that are not in an arrow, which is done in $O(m_p + m_{p+1})$ time.
Then, $\partial_p$ is constructed by taking the rows and columns of $P_p$ that correspond to critical cells and converting their entries to $\Z/2\Z$.
We may then compute persistent homology as normal, performing the matrix reduction algorithm on each boundary matrix, with a cost of $O(C^3)$, where $C$ is the total number of critical cells.

In theory, on an arbitrary simplicial complex, this incurs the standard cost of Algorithm \ref{alg:MatrixReduction}.
In practice, though, this lowers the size of boundary matrices: 
each arrow of $F_f$ removes a row and column from a boundary matrix.
Classifying complexes that have bounds on their number of critical simplices is a direction for future work.

\subsection{Matrix Updates for Vertex Swaps}

When a vertex swap has occurred where the order function changes from $f$ to $f'$, $F_f$ must be updated to $F_{f'}$.
As this happens, the collection of path matrices $\PP$ will update along with our vector field.
Some arrows of $F_f$ must be deleted while new ones must be added, and we must update $\PP$ accordingly.

\begin{lemma}
    Suppose the arrow $\alpha^{(p)} \rightarrow \beta^{(p+1)}$ is deleted from $F_f$ naturally.
    Then no new $F_f$-paths ending at $\alpha$ or starting at $\beta$ are deleted. 
    Additionally, for any $\tau^{(p+1)} \neq \beta$ and $\sigma^{(p)} \neq \alpha$,
    the number of $F_f$-paths from $\tau$ to $\sigma$ decreases by 
    $\Gamma(\tau, \alpha) \Gamma(\beta, \sigma)$.
    \end{lemma}
\begin{proof}
    Let $\tilde{F} \coloneq F_f - \{\alpha \rightarrow \beta \}$.
    This deletion occurs naturally by assumption; therefore, the addition of $\alpha \rightarrow \beta$ to $\tilde{F}$ occurs naturally as well.
    The path matrix of $F_f$ must be the same after the deletion and addition of this arrow.
    Therefore, the change in the number of paths must exactly match that of Prop. \ref{prop:countingAddition}. 
\end{proof}

We therefore wish for the change from $F_f$ to $F_{f'}$ to occur via natural additions and deletions so that we know how $\PP$ changes.
Since $F_f \cap F_{f'}$ is a subfield of both $F_f$ and $F_{f'}$, we can use that as a midpoint.
We delete all the unnecessary arrows of $F_f$, which occurs naturally, and then add arrows naturally until we reach $F_{f'}$.
Therefore, for implementation as in Algorithm \ref{alg:covf_update}, rather than updating $F_f$ directly, we advise keeping a list of arrows that must be deleted and a list of arrows that must be added.
Once all of $\overline \St(x) \cap \overline \St(y)$ has been accounted for, delete all necessary arrows, updating $\PP$ for each one.
Afterwards, add all of the new arrows, again updating $\PP$ each time.
The above process is $O(|\overline \St(x) \cap \overline \St(y)| M^2)$, as each $\sigma \in \overline \St(x) \cap \overline \St(y)$ may contribute to both an addition and a deletion (recall that $M \coloneq \max_p\{ m_p\}$).

\section{Code and Results}


\begin{figure}
\centering
\begin{subfigure}{.5\textwidth}
  \centering
  \includegraphics[width=0.8\linewidth]{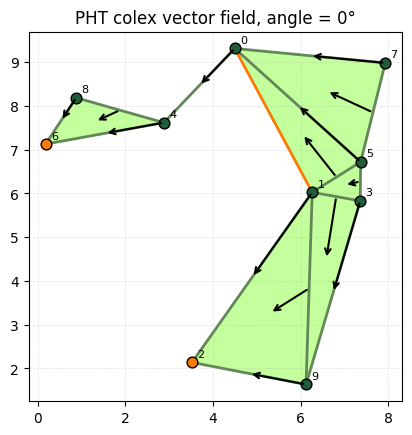}
  \label{fig:sub1}
\end{subfigure}%
\begin{subfigure}{.5\textwidth}
  \centering
  \includegraphics[width=0.8\linewidth]{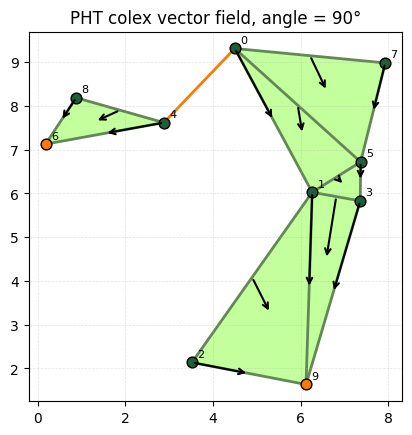}
  \label{fig:sub2}
\end{subfigure}
\captionsetup{width=.8\linewidth}
\caption{Example of a randomly generated alpha complex with colex vector field induced by the height function from different angles. Critical simplices are displayed in orange.}
\label{fig:examples}
\end{figure}


We provide proof-of-concept code, publicly available \href{https://github.com/kwoytowich/Lower-Star-Vineyard-Updates-with-Discrete-Morse-Theory.git}{here}. It is written in Python and builds off of ceREEBerus's LowerStar structure \cite{cereeberus}, a version of Gudhi's SimplexTree \cite{gudhi:FilteredComplexes} with tools to update filtration values for lower star filtrations.
We also include a demonstration of its runtime against a more standard method of computing persistence by Algorithm \ref{alg:MatrixReduction}.

In the setting of the Persistent Homology Transform (PHT), each direction on $\mathbb{S}^{d-1}$ yields a height function on $V$.
Apart from the finitely many directions for which two vertices have the same height, this in turn induces an order function $f$ by ranking the vertices.
The sphere can be broken up into finitely many strata where the order function remains constant \cite{Curry2022}.
Thus, one could traverse around the sphere, and for each order function, compute $F_f$ and thus the persistent homology from scratch.
However, if we know how the order function changes, we can update the colex vector field more directly.

We perform the PHT on the alpha complex of a point cloud of $n$ points in $\R^2$, randomly sampled from a uniform distribution.
A complex with $n$ vertices in general position stratifies $\mathbb{S}^1$ into $n^2 - n$ parts, each with a unique order function on the vertices.
Traversing between these strata causes the order function to change via an order-adjacent swap.
Our implementation processes one vertex swap at a time, computing $n^2 - n$ individual persistence diagrams.
We do this using both the colex vector field approach and a standard persistence reduction approach, where each direction constructs a boundary matrix from the colex order and computes the pairs.
The non-diagonal pairs of these two approaches are equivalent, but the extra machinery of the colex vector field means that many diagonal pairs are not calculated, though they can be recovered by the arrows in the field.
We compare the runtime of the two approaches below.

\begin{table}[h!]
\centering
\begin{tabular}{|l|l|l|l|l|l|l|}
\hline
         & $n = 15$ & $n = 15$ & $n = 20$ & $n = 20$ & $n = 25$ & $n = 25$ \\ 
         & $r = 3$ & $r = 6$ & $r = 3$ & $r = 6$ & $r = 3$ & $ r = 6$ \\ \hline
Colex (in sec.)   &    0.3021    &    0.3080      &   1.2099     &   1.2227     &    3.5588      &     3.5771      \\ \hline
Standard (in sec.)  &   0.2982    &   0.4026     &   1.1082    &    1.4235       &      2.9858    &     3.6419      \\ \hline
Avg.~no.~crit. cells &  4.5599      &  3.0229      &    5.1242    &   3.5461      &     5.4912           &    3.6684             \\ \hline
\end{tabular}
\captionsetup{width=.8\linewidth}
\caption{The time taken for computing the PHT using the Morse colex update method presented, along with the standard method (pointwise application of Alg.~\ref{alg:MatrixReduction}) for alpha complexes computed on a $n$ randomly sampled points using threshold parameter $r$.}
\label{table:times}
\end{table}
In Table \ref{table:times}, the entries of the Colex and Standard rows are the average runtimes in seconds over 100 random alpha complexes, with $n$ points and threshold $r$. 
For each complex, the $n$ points are chosen from a uniform distribution over $(0, 10) \times (0, 10) \subset \R^2$.
In Figure \ref{fig:examples}, one such generated complex is depicted, and the colex vector field associated with two different height functions are overlaid.
Each column has its own fixed set of 100 complexes, and the bottom row computes the average number of critical cells per stratum over those generated complexes. 
These computations were performed on an AMD Ryzen 5 processor with 32 GB of RAM.

\section{Discussion}
In this paper, we formalize the notion of a discrete vector field for a simplicial filtration defined by an order function on the vertices and equipped with colexicographic ordering. We demonstrate that the colex vector field gives rise to smaller boundary matrices that only take account of a subset of simplices within the complex, so-called critical simplices. We use the fact that the persistent homology summary given by this colex vector field coincides with the ordinary persistent homology to compute persistence in this setting, thereby yielding improved computational efficiency at least in practice. We implement an efficient procedure for updating the colex vector field and, consequently, the path and boundary matrices, when adjacent vertices swap in the filtration. We show that our method not only produces the same persistence diagram, but also yields the same persistence pairs for the lower star filtration given by the standard persistence algorithm. Some pairs do not get calculated by the Morse complex, but they all correspond to arrows and can be recovered for diagonal points.  

We also present some complexity analysis and code implementation to demonstrate the practicability of our discrete Morse theory-inspired method of computing and updating persistent homology. Our experimental results show practical improvements in computation time as compared with the standard persistence reduction algorithm, even though we cannot guarantee that asymptotic runtimes are better. The colex vector field approach performs well in complexes with a higher threshold value $r$, corresponding to more complexes and higher connectivity. In these complexes, the number of critical cells is lower than the total number of simplices, resulting in smaller boundary matrices on which reduction needs to be done. This also means that better implementations could yield better speedups. While these experiments are encouraging, a comprehensive evaluation comparing against optimized implementations and large-scale datasets is left for future work. Potential avenues for improving on the code include additional optimization and a C\texttt{++} port.

\section*{Declarations}
The authors have no competing interests to declare that are relevant to the content of this article.


    




\nocite{*}
\printbibliography
\end{document}